\documentclass[11pt,oneside,english]{amsart}
\usepackage[T1]{fontenc}
\usepackage[latin9]{inputenc}
\usepackage{geometry}
\usepackage{amstext}
\usepackage{amsthm}
\usepackage{amssymb}
\usepackage{setspace}
\makeatletter

\theoremstyle{plain}
\newtheorem{thm}{\protect\theoremname}[section]
\theoremstyle{plain}
\newtheorem{lem}[thm]{\protect\lemmaname}
\theoremstyle{remark}
\newtheorem{rem}[thm]{\protect\remarkname}
\theoremstyle{plain}
\newtheorem{prop}[thm]{\protect\propositionname}

\makeatother

\usepackage{babel}
\providecommand{\lemmaname}{Lemma}
\providecommand{\propositionname}{Proposition}
\providecommand{\remarkname}{Remark}
\providecommand{\theoremname}{Theorem}

\begin{document}

\title{berry-esseen type estimate and return sequence for parabolic iteration
in the upper half-plane}

\author{octavio arizmendi, mauricio salazar, 
jiun-chau wang}

\date{November 27, 2018}
\begin{abstract}
We answer the question of finding a Berry-Esseen type theorem for
the convergence rate in monotone central limit theorem. When the underlying
measure is singular to Lebesgue measure, this central limit process
is viewed as an infinite measure-preserving dynamical system and we
prove that it has a regularly varying return sequence of index $1/2$.
The main tool in our proofs is the complex analysis techniques from
free probability.
\end{abstract}



\address{O. Arizmendi and M. Salazar: Centro de Investigaci\'{o}n en Matem\'{a}ticas, Apdo. Postal
402, Guanajuato, Gto. 36000, Mexico; J.-C. Wang: Department of Mathematics
and Statistics, University of Saskatchewan, Saskatoon, Saskatchewan
S7N 5E6, Canada}

\email{octavius@cimat.mx; maurma@cimat.mx; jcwang@math.usask.ca}
\maketitle

\section{Introduction}

Let $\mathbb{C}^{+}=\left\{ x+iy:y>0\right\} $, and let $F:\mathbb{C}^{+}\rightarrow\mathbb{C}^{+}$
be an analytic map satisfying $F(iy)/iy\rightarrow1$ as $y\rightarrow\infty$.
This paper studies the dynamics of the parabolic iteration $\left\{ F^{\circ n}\right\} _{n=1}^{\infty}$
from two different perspectives. 

First, by Proposition 5.1 in \cite{BVunbdd}, there is a one-to-one
correspondence between the set of all analytic self-maps $F$ on $\mathbb{C}^{+}$
with $\lim_{y\rightarrow\infty}F(iy)/iy=1$ and that of all probability
measures $\mu$ on $\mathbb{R}$ in the sense that 
\begin{equation}
\frac{1}{F(z)}=\int_{\mathbb{R}}\frac{1}{z-t}\,d\mu(t),\quad z\in\mathbb{C}^{+}.\label{Fdef}
\end{equation}
Accordingly, we will denote the function $F$ by $F_{\mu}$ and call
it the \emph{$F$-transform} of $\mu$ from now on. Moreover, by Muraki's
monotone probability theory \cite{Muraki1,Muraki2}, each $n$-fold
iteration $F_{\mu}^{\circ n}$ corresponds to a unique probability
measure $\mu^{\triangleright n}$ (called the \emph{$n$-th monotone
convolution power} of $\mu$) on $\mathbb{R}$ such that 
\begin{equation}
\frac{1}{F_{\mu}^{\circ n}(z)}=\int_{\mathbb{R}}\frac{1}{z-t}\,d\mu^{\triangleright n}(t),\quad z\in\mathbb{C}^{+},\quad n\geq1.\label{MonotonePowers}
\end{equation}
Indeed, the measure $\mu^{\triangleright n}$ may be interpreted as
the spectral distribution of a sum of $n$ self-adjoint operators
that are monotonically independent and identically distributed according
to the measure $\mu$. (See Section 2.1 for a review of the construction
of monotonically independent random variables.) Of course, at the
level $n=1$ one has $\mu^{\triangleright1}=\mu$ and the equation
(\ref{MonotonePowers}) reduces to (\ref{Fdef}), the definition of
$F_{\mu}$. 

We assume the measure $\mu$ in (\ref{Fdef}) has \emph{zero mean}
and \emph{unit variance} throughout this paper. Let $\mu_{n}$ be
the dilation defined by $\mu_{n}(A)=\mu^{\triangleright n}(\sqrt{n}A)$
for all sets $A$ in the Borel $\sigma$-field $\mathcal{B}$ on $\mathbb{R}$.
Then the measures $\mu_{n}$ converge weakly on $\mathbb{R}$ to the
absolutely continuous probability measure 
\[
d\gamma(t)=\frac{1}{\pi\sqrt{2-t^{2}}}\,d\lambda(t),\quad-\sqrt{2}<t<\sqrt{2},
\]
where $\lambda$ denotes the Lebesgue measure on $\mathbb{R}$. This
is known as the \emph{monotone central limit theorem} (Monotone CLT)
\cite{Muraki2}. Since weak convergence to an absolutely continuous
measure implies the uniform convergence of the corresponding cumulative
distribution functions, it is natural to prove the following quantitative
version of Monotone CLT:
\begin{thm}
We write $F_{\mu}$ in its Nevanlinna integral representation 
\[
F_{\mu}(z)=z+\int_{\mathbb{R}}\frac{1}{t-z}\,d\nu(t),\qquad z\in\mathbb{C}^{+}.
\]
If the measure $\nu$ has the first absolute moment $c\in(0,+\infty)$,
then one has the Kolmogorov distance 
\[
\sup_{x\in\mathbb{R}}\left|\mu_{n}\left((-\infty,x])\right)-\gamma\left((-\infty,x])\right)\right|\leq71\sqrt[4]{c}\,n^{-1/8}
\]
for sufficiently large $n$.
\end{thm}

This result is an analogue of the Berry-Esseen estimate for the classical
CLT. Here the first absolute moment $c$ is acting like the third absolute moment in the classical situation, and
$c<+\infty$ if $\mu$ is further assumed to have finite fourth moment.
Unlike the classical Berry-Esseen theorem, the best convergence rate
here is at the order of $O(n^{-1/4})$ under a finite sixth moment condition and it cannot be improved without additional assumptions on the measure $\mu$ (see Theorem 3.1 and Remark 3.2). The case of $c=0$ (i.e., $\nu=\delta_{0}$, the point mass at zero) corresponds to $\mu=(\delta_{-1}+\delta_{1})/2$
and the convergence rate is also $O(n^{-1/4})$. 

Our second perspective on $\left\{ F_{\mu}^{\circ n}\right\} _{n=1}^{\infty}$
comes from the fact that when $\mu$ is singular to the Lebesgue measure
$\lambda$ on $\mathbb{R}$, the boundary limit
\[
T_{\mu}x=\lim_{y\rightarrow0^{+}}F_{\mu}(x+iy)
\]
exists in $\mathbb{R}$ for $\lambda$-almost every $x\in\mathbb{R}$
and the transformation $T_{\mu}$ preserves the measure $\lambda$
in the sense that $\lambda\circ T_{\mu}^{-1}=\lambda$ on the $\sigma$-field
$\mathcal{B}$ \cite{Letac}. Hence the dynamical system $\left(\mathbb{R},\mathcal{B},T_{\mu}\right)$
is an object of study in infinite ergodic theory \cite{AaronsonBook}. 

It was proved in \cite{Wang} that $T_{\mu}$ is \emph{conservative}
in the sense that for any $A\in\mathcal{B}$ with $\lambda(A)>0$,
one has the limit 
\[
S_{n}^{A}(x)=\#\left\{ k\in\mathbb{Z}\cap[0,n-1]:T_{\mu}^{\circ k}x\in A\right\} \rightarrow+\infty\quad(n\rightarrow\infty)
\]
for $\lambda$-a.e. $x\in A$. In other words, $T_{\mu}$ has the
recurrence behavior that any typical orbit $\{T_{\mu}^{\circ n}x\}_{n=1}^{\infty}$
starting from $x\in A$ will return to $A$ infinitely often. 

Given such a nature of the system $\left(\mathbb{R},\mathcal{B},T_{\mu}\right)$, it is reasonable to investigate the asymptotic growth rate for the
occupation times $S_{n}^{A}(x)$, and to ask further how this growth
rate may depend on the set $A$, the point $x$, or the underlying
measure $\mu$ (cf. \cite{AaronsonBook}).

Aaronson proved that a normalization sequence for $S_{n}^{A}$, or, more generally, for the ergodic sum 
\[
S_{n}(f)=\sum_{k=0}^{n-1}f\circ T_{\mu}^{\circ k}
\]
of any $f\in L_{1}(\lambda)$ exists only at the \emph{dual} level
\cite{AaronsonPaper}. His result implies that there are constants
$a_{n}(T_{\mu})>0$ satisfying a weak form of Birkhoff type ergodic theorem:
\[
\frac{1}{a_{n}(T_{\mu})}\sum_{k=0}^{n-1}\widehat{T_{\mu}}^{k}f\rightarrow\int_{\mathbb{R}}f\,d\lambda\quad\text{a.e. as }n\rightarrow\infty,\quad\forall f\in L_{1}(\lambda),
\]
where the Perron-Frobenius operator $\widehat{T_{\mu}}:L_{1}(\lambda)\rightarrow L_{1}(\lambda)$
is defined through the duality 
\[
\int_{\mathbb{R}}\widehat{T_{\mu}}f\cdot g\,d\lambda=\int_{\mathbb{R}}f\cdot g\circ T_{\mu}\,d\lambda,\quad f\in L_{1}(\lambda),\;g\in L_{\infty}(\lambda).
\]
The sequence $a_{n}(T_{\mu})$ is called a \emph{return sequence}
of $T_{\mu}$ and is unique up to asymptotic equivalence. Aaronson
also showed that $a_{n}(T_{\mu})\sim\pi^{-1}\sqrt{2n}$ when $\mu$
has bounded support (see also \cite{LetacM}). We obtain an extension
of this result, saying that the $\pi^{-1}\sqrt{2n}$ growth rate is
in fact universal in the set of measures attracted to the arcsine
law $\gamma$ in Monotone CLT. (This set is the same as the normal domain
of attraction of the standard Gaussian law in the classical CLT \cite{Wang}).
\begin{thm}
We have $a_{n}(T_{\mu})\sim\pi^{-1}\sqrt{2n}$ among all singular distributions
$\mu$ with zero mean and unit variance.
\end{thm}

In particular, the Darling-Kac type theorem from \cite[Corollary 3.7.3]{AaronsonBook}
implies that for any absolutely continuous probability measure $P$
on $\mathbb{R}$ and for any integrable $f:\mathbb{R}\rightarrow[0,+\infty]$
with $\left\Vert f\right\Vert _{1}=1$, one has the weak limit
\[
\lim_{n\rightarrow\infty}P\left(\left[S_{n}(f)\leq\pi^{-1}\sqrt{2n}t\right]\right)=\frac{2}{\pi}\int_{0}^{t}e^{-x^{2}/\pi}\,dx,\quad t\geq0.
\]
We then obtain a good understanding for the growth of the random variable $S_{n}^{A}=S_{n}(I_{A})$
in the distributional sense. (For the pointwise behavior of $S_{n}^{A}$, we refer to the book \cite{AaronsonBook} for various rational ergodic theorems.)

For the organization of this paper, we first collect the preliminary materials in Section 2, then we prove the Berry-Esseen estimate and the asymptotics of $a_{n}(T_{\mu})$ in Sections 3 and 4, respectively. 

\section{Preliminaries}

\subsection{Realization of monotone convolution and $F$-transform}
Given two probability measures $\nu_{1}$ and $\nu_{2}$ on $\mathbb{R}$,
we consider the Hilbert space $H=L_{2}\left(\mathbb{R}^{2},\nu_{1}\otimes\nu_{2}\right)$
and two self-adjoint operators $X_{1}$ and $X_{2}$ on $H$ defined
by 
\[
X_{1}h_{1}(x,y)=x\int_{\mathbb{R}}h_{1}(x,t)\,d\nu_{2}(t),\quad X_{2}h_{2}(x,y)=yh_{2}(x,y),
\]
where the domain of $X_{1}$ consists of all $h_{1}\in H$ satisfying
\[
\int_{\mathbb{R}}x^{2}\left|\int_{\mathbb{R}}h_{1}(x,t)\,d\nu_{2}(t)\right|^{2}\,d\nu_{1}(x)<+\infty,
\]
and that of $X_{2}$ consists of all $h_{2}\in H$ such that
\[
\int_{\mathbb{R}^{2}}\left|yh_{2}(x,y)\right|^{2}\,d\nu_{1}\otimes\nu_{2}(x,y)<+\infty.
\]

It follows that the sum $X_{1}+X_{2}$ is essentially self-adjoint,
and hence the spectral theorem yields a unique probability measure
$\nu_{1}\triangleright\nu_{2}$ on $\mathbb{R}$ such that 
\[
\varphi\left(\psi\left(X_{1}+X_{2}\right)\right)=\int_{\mathbb{R}}\psi\,d\nu_{1}\triangleright\nu_{2}
\]
for all bounded Borel functions $\psi$ on $\mathbb{R}$. That is,
$\nu_{1}\triangleright\nu_{2}$ is the \emph{distribution} of the
(noncummutative) random variable $X_{1}+X_{2}$. Here $\varphi$ denotes
the vector state associated with the constant function one on $\mathbb{R}^{2}$,
and $\psi\left(X_{1}+X_{2}\right)$ is defined through the functional
calculus. 

The measure $\nu_{1}\triangleright\nu_{2}$ is called the \emph{monotone
convolution} of $\nu_{1}$ and $\nu_{2}$. It was shown in \cite{Franz}
that the algebras $\mathcal{A}_{i}=\left\{ f(X_{i}):f\in C_{b}(\mathbb{R}),\:f(0)=0\right\} $,
$i\in\left\{ 1,2\right\} $, are \emph{monotonically independent}
in the $C^{*}$-probability space $\left(B(H),\varphi\right)$ in
the sense that for every mixed moment $\varphi\left(a_{1}a_{2}\cdots a_{n}\right)$
where $a_{j}\in\mathcal{A}_{i_{j}}$, $i_{j}\in\{1,2\}$, and $i_{1}\neq i_{2}\neq\cdots\neq i_{n}$,
one has that 
\[
\varphi\left(a_{1}a_{2}\cdots a_{n}\right)=\varphi(a_{j})\varphi\left(a_{1}\cdots a_{j-1}a_{j+1}\cdots a_{n}\right),
\]
whenever $a_{j}\in\mathcal{A}_{2}$. 

An important feature of the monotone convolution is that 
\[
F_{\nu_{1}\triangleright\nu_{2}}(z)=F_{\nu_{1}}\circ F_{\nu_{2}}(z),\quad z\in\mathbb{C}^{+}.
\]
Thus, the existence of the monotone convolution powers $\mu^{\triangleright n}$
may be proved by applying the above construction inductively to the
measure $\mu$. To reiterate, we now have that $\mu^{\triangleright1}=\mu$,
$\mu^{\triangleright n}=\mu\triangleright\mu^{\triangleright(n-1)}$
for $n\geq2$, and $F_{\mu^{\triangleright n}}=F_{\mu}^{\circ n}$
in $\mathbb{C}^{+}$. As for the dilation $\mu_{n}$, it follows from
the spectral theorem that 
\begin{equation}
F_{\mu_{n}}(z)=\frac{1}{\sqrt{n}}F_{\mu^{\triangleright n}}\left(\sqrt{n}z\right)=\frac{1}{\sqrt{n}}F_{\mu}^{\circ n}\left(\sqrt{n}z\right),\quad z\in\mathbb{C}^{+}.\label{dilation}
\end{equation}

We also refer to Muraki's original papers \cite{Muraki1,Muraki2}
for a more general monotone product construction, by which a sequence
$X_{1},X_{2},\cdots,X_{n}$ of monotonically independent self-adjoint
operators with the same distribution $\mu$ can be constructed at
once on a Hilbert space so that the measure $\mu^{\triangleright n}$
is the distribution of the sum $X_{1}+X_{2}+\cdots+X_{n}$. 

We conclude this subsection by gathering some well-known facts about
the $F$-transform. First, $\mu$ has finite second moment if and
only if the Nevanlinna integral form of $F_{\mu}$ is given by 
\begin{equation}
F_{\mu}(z)=z-m(\mu)+\int_{\mathbb{R}}\frac{1}{t-z}\,d\nu(t),\quad z\in\mathbb{C}^{+},\label{Nevanlinna}
\end{equation}
where $m(\mu)$ denotes the mean of $\mu$ and $\nu$ is a finite
Borel measure on $\mathbb{R}$ with $\nu(\mathbb{R})=var(\mu)$, the
variance of $\mu$. 

We also note that if the measure $\mu$ has finite moment $m_{2n}(\mu)$
of order $2n$ then $m_{2n-2}(\nu)<+\infty$, see \cite{Akhiezer}. 

The integral form (\ref{Nevanlinna}) shows that $\Im F_{\mu}(z)\geq\Im z$
for $z\in\mathbb{C}^{+}$, and the equality holds at some point in
$\mathbb{C}^{+}$ if and only if $\mu=\delta_{a}$ for some $a\in\mathbb{R}$.

\subsection{Bai's inequality and the arcsine measure}

For notational convenience, we use $dx$ or $dt$ to denote the Lebesgue
measure $\lambda$ in the sequel. 

Let $C_{\mu}(x)=\mu\left((-\infty,x]\right)$, $x\in\mathbb{R}$,
be the cumulative distribution function of a probability measure $\mu$.
Recall the \emph{Bai's inequality }from \cite{Bai} as follows: If
two probability measures $\mu$ and $\nu$ satisfy
\[
\int_{\mathbb{R}}\left|C_{\mu}(x)-C_{\nu}(x)\right|\,dx<+\infty,
\]
then one has the supremum norm \[\left\Vert C_{\mu}-C_{\nu}\right\Vert _{\infty} \leq \int_{\mathbb{R}}\left|\frac{1}{F_{\mu}(x+iy)}-\frac{1}{F_{\nu}(x+iy)}\right|\,dx+\frac{1}{y}\sup_{r\in\mathbb{R}}\int_{\left|t\right|\leq2\sqrt{3}y}\left|C_{\nu}(r+t)-C_{\nu}(t)\right|\,dt\]
for all $y>0$. We examine this inequality when $\nu$ is the arcsine
measure $\gamma$. 
\begin{lem}
For any probability measure $\mu$ with finite variance and for any
$0<y<(4\sqrt{3})^{-1}$, one has the Kolmogorov distance 
\[
\left\Vert C_{\mu}-C_{\gamma}\right\Vert _{\infty}\leq\int_{\mathbb{R}}\left|\frac{1}{F_{\mu}(x+iy)}-\frac{1}{F_{\gamma}(x+iy)}\right|\,dx+11\sqrt{y}.
\]
\end{lem}

\begin{proof}
We first show that $\mu$ and $\gamma$ satisfy the integrability
requirement in Bai's inequality. Indeed, the Chebyshev's inequality
yields
\begin{eqnarray*}
\int_{\mathbb{R}}\left|C_{\mu}(x)-C_{\gamma}(x)\right|\,dx & \leq & \int_{\left|x\right|\geq\sqrt{2}}\mu\left(\left\{ t:\left|t\right|\geq x\right\} \right)\,dx+\int_{-\sqrt{2}}^{\sqrt{2}}\left|C_{\mu}(x)-C_{\gamma}(x)\right|\,dx\\
 & \leq & var(\mu)\int_{\left|x\right|\geq\sqrt{2}}x^{-2}\,dx+4\sqrt{2}<+\infty.
\end{eqnarray*}

Next, we show 
\begin{equation}
\frac{1}{y}\sup_{r\in\mathbb{R}}\int_{\left|t\right|\leq2\sqrt{3}y}\left|C_{\gamma}(r+t)-C_{\gamma}(t)\right|\,dt\leq11\sqrt{y}\label{Bai}
\end{equation}
to conclude the proof of this lemma. 

We first examine the case $0\leq t\leq2\sqrt{3}y$, in which our assumption
on $y$ implies $t<\sqrt{2}-t$. Note that the difference $\Delta C(r)=C_{\gamma}(r+t)-C_{\gamma}(t)$
satisfies $\Delta C(r)=0$ for $r\in(-\infty,-\sqrt{2}-t]\cup[\sqrt{2},+\infty)$,
because the arcsine density $f(x)=\pi^{-1}(2-x^{2})^{-1/2}$ is supported
on $(-\sqrt{2},\sqrt{2})$. Also, from the graph of $f$, we observe
that 
\[
\Delta C(r)\leq\Delta C(\sqrt{2}-t)=\Delta C(-\sqrt{2})=\int_{\sqrt{2}-t}^{\sqrt{2}}f(x)\,dx
\]
 for $r\in(-\sqrt{2}-t,-\sqrt{2}]\cup[\sqrt{2}-t,\sqrt{2})$. 

If $r\in[0,\sqrt{2}-t)$, we have the derivative $d\Delta C(r)/dr>0$,
and therefore the monotonicity of $\Delta C$ implies that $\Delta C(r)\leq\Delta C(\sqrt{2}-t)$
for such $r$. For $r\in(-\sqrt{2},-t]$, the symmetry of $\Delta C$
shows that $\Delta C(r)=\Delta C(-r-t)\leq\Delta C(\sqrt{2}-t)$ as
well. 

It remains to estimate $\Delta C(r)$ for $r\in(-t,0)$. In this case
we note that 
\begin{eqnarray*}
\Delta C(r) & = & \int_{r}^{0}f(x)\,dx+\int_{0}^{r+t}f(x)\,dx\\
 & \leq & \int_{-t}^{0}f(x)\,dx+\int_{0}^{t}f(x)\,dx\\
 & = & \Delta C(-t)+\Delta C(0)\leq2\Delta C(\sqrt{2}-t)
\end{eqnarray*}
by our previous discussions. 

We have shown that $\Delta C(r)\leq2\Delta C(\sqrt{2}-t)$ when $0\leq t\leq2\sqrt{3}y$.
By the same method, one can also prove that $\left|\Delta C(r)\right|\leq2\Delta C(\sqrt{2}+t)$
for $0>t\geq-2\sqrt{3}y$. In summary, we have $\left|\Delta C(r)\right|\leq2\left[C(\sqrt{2})-C(\sqrt{2}-\left|t\right|)\right]$
for any $r\in\mathbb{R}$ and $\left|t\right|\leq2\sqrt{3}y$. 

For $r\in\mathbb{R}$, we now compute 
\begin{eqnarray*}
\int_{\left|t\right|\leq2\sqrt{3}y}\left|\Delta C(r)\right|\,dt & \leq & 2\int_{\left|t\right|\leq2\sqrt{3}y}C(\sqrt{2})-C(\sqrt{2}-\left|t\right|)\,dt\\
 & = & 4\int_{0}^{2\sqrt{3}y}C(\sqrt{2})-C(\sqrt{2}-t)\,dt\\
 & = & 4\int_{0}^{2\sqrt{3}y}\int_{\sqrt{2}-t}^{\sqrt{2}}\frac{1}{\pi\sqrt{2-x^{2}}}\,dx\,dt\\
 & \leq & \frac{4}{\pi\sqrt[4]{2}}\int_{0}^{2\sqrt{3}y}\int_{\sqrt{2}-t}^{\sqrt{2}}\frac{1}{\sqrt{\sqrt{2}-x}}\,dx\,dt\\
 & = & \frac{4}{\pi\sqrt[4]{2}}\int_{0}^{2\sqrt{3}y}\int_{0}^{t}\frac{1}{\sqrt{x}}\,dx\,dt=\frac{32\sqrt[4]{2}}{\pi\sqrt[4]{3}}y\sqrt{y},
\end{eqnarray*}
from which the estimate (\ref{Bai}) follows. 
\end{proof}
The $F$-transform of $\gamma$ is known to be 
\[
F_{\gamma}(z)=\sqrt{z^{2}-2},\quad z\in\mathbb{C}^{+},
\]
where the branch of the square root is chosen to be analytic on $\mathbb{C}\setminus[0,+\infty)$
and $\sqrt{-1}=i$. Our next lemma deals with small, uniform perturbations
of $F_{\gamma}$ in Bai's inequality.
\begin{lem}
If $y>0$ is given and $\varepsilon$ is a continuous function
on the horizontal line $L=\left\{ x+iy:x\in\mathbb{R}\right\} $ such
that $\left|\varepsilon(z)\right|<3y/2$ for all $z=x+iy\in L$, then 
\[
\int_{\mathbb{R}}\left|\frac{1}{\sqrt{z^{2}-2+\varepsilon(z)}}-\frac{1}{\sqrt{z^{2}-2}}\right|\,dx\leq39\sqrt{\Im z},\quad z\in L.
\]
\end{lem}

\begin{proof}
The function $z\mapsto z^{2}-2$ maps
the horizontal line $L$ bijectively to the parabola $\left\{ (u,v):u=v^{2}/4y^{2}-2-y^{2}\right\} $. For all $z\in L$, let $\ell(z)$ be the line segment connecting the points $z^{2}-2$
and $z^{2}-2+\varepsilon(z)$. Our hypothesis on
$\varepsilon(z)$ implies that $\ell(z)$ is contained in $\mathbb{C}\setminus[0,+\infty)$,
and hence the fundamental theorem of calculus for complex line integral
yields
\[
\sqrt{z^{2}-2+\varepsilon(z)}-\sqrt{z^{2}-2}=\int_{w\in\ell(z)}\frac{1}{2\sqrt{w}}\,dw.
\]

On account of the parametrization $w(t)=$$z^{2}-2+t\varepsilon(z)$,
$t\in[0,1]$, and the fact
\[
\left|z^{2}-2+t\varepsilon(z)\right|\geq\left|z^{2}-2\right|-\frac{3y}{2}>\left|z^{2}-2\right|-\frac{3\left|z^{2}-2\right|}{4}=\frac{\left|z^{2}-2\right|}{4},
\]
we get the estimate
\begin{eqnarray*}
\left|\int_{w\in\ell(z)}\frac{1}{2\sqrt{w}}\,dw\right| & \leq & \frac{\text{length}(\ell(z))}{2}\cdot\sup_{w\in\ell(z)}\frac{1}{\sqrt{\left|w\right|}}\\
 & < & \frac{\left|\varepsilon(z)\right|}{\left|z^{2}-2\right|^{1/2}}<\frac{3y}{2\left|z^{2}-2\right|^{1/2}},
\end{eqnarray*}
which leads further to the inequality
\[
\left|\frac{1}{\sqrt{z^{2}-2+\varepsilon(z)}}-\frac{1}{\sqrt{z^{2}-2}}\right|\leq\frac{3y}{\left|z^{2}-2\right|^{3/2}}.
\]

Therefore, we conclude that 
\begin{eqnarray*}
\int_{\mathbb{R}}\left|\frac{1}{\sqrt{z^{2}-2+\varepsilon(z)}}-\frac{1}{\sqrt{z^{2}-2}}\right|\,dx & \leq & \int_{\mathbb{R}}\frac{3y}{\left|z^{2}-2\right|^{3/2}}\,dx\\
 & = & \int_{0}^{\infty}\frac{6y}{\left|z^{2}-2\right|^{3/2}}\,dx\\
 & = & \int_{0}^{2}\frac{6y}{\left|z+\sqrt{2}\right|^{3/2}\left|z-\sqrt{2}\right|^{3/2}}\,dx\\
 &  & +\int_{2}^{\infty}\frac{6y}{\left|z^{2}-2\right|^{3/2}}\,dx\\
 & \leq & \int_{0}^{2}\frac{6y}{\left|z-\sqrt{2}\right|^{3/2}}\,dx+\int_{2}^{\infty}\frac{6y}{(x-1)^{3}}\,dx\\
 & = & \int_{-\sqrt{2}}^{2-\sqrt{2}}\frac{6y}{(x^{2}+y^{2})^{3/4}}\,dx+3y\\
 & \leq & 6\sqrt{y}\int_{\mathbb{R}}\frac{1}{(u^{2}+1)^{3/4}}\,du+3\sqrt{y}<39\sqrt{y}.
\end{eqnarray*}
\end{proof}

\section{Berry-Esseen Estimates}

Recall that $\mu_{n}$ is the normalization of $\mu^{\triangleright n}$
in Monotone CLT. From the characterization of the domain of attraction
of $\gamma$ in \cite{Wang}, each monotone convolution power $\mu^{\triangleright n}$
has variance $n$ and zero mean, implying further that $var(\mu_{n})=1$
and $m(\mu_{n})=0$. In particular, $\mu_{n}$ satisfies the integrability
hypothesis of Bai's inequality from the proof of Lemma 2.1. 

By (\ref{dilation}) and (\ref{Nevanlinna}), we have 
\[
F_{n}(z)=\frac{F_{\mu}(\sqrt{n}z)}{\sqrt{n}}=z+\frac{1}{\sqrt{n}}\int_{\mathbb{R}}\frac{1}{t-\sqrt{n}z}\,d\nu(t),\quad z\in\mathbb{C}^{+},
\]
where $\nu(\mathbb{R})=var(\mu)=1$, and $F_{\mu_{n}}=F_{n}^{\circ n}$
in $\mathbb{C}^{+}$. Note that $\Im F_{n}(z)>\Im z$ for $z\in\mathbb{C}^{+}$.

We introduce two conjugacy functions $\psi_{1}(z)=z^{2}$ and $\psi_{2}(z)=\sqrt{z}$,
as well as the auxiliary function 
\[
r_{n}(z)=\frac{2}{n}\int_{\mathbb{R}}\frac{t}{t-\sqrt{n}z}\,d\nu(t)+\frac{1}{n}\left[\int_{\mathbb{R}}\frac{1}{t-\sqrt{n}z}\,d\nu(t)\right]^{2},\quad z\in\mathbb{C}^{+},
\]
so that 
\begin{eqnarray}
F_{\mu_{n}}(\sqrt{z})^{2}=\psi_{1}\circ F_{\mu_{n}}\circ\psi_{2}(z) & = & (\psi_{1}\circ F_{n}\circ\psi_{2})^{\circ n}(z)\label{iteration}\\
 & = & \left[z-\frac{2}{n}+r_{n}(\sqrt{z})\right]^{\circ n}\nonumber \\
 & = & z-2+\sum_{j=0}^{n-1}r_{n}\left(F_{n}^{\circ j}(\sqrt{z})\right)\nonumber 
\end{eqnarray}
for any $z\in\mathbb{C}\setminus\mathbb{R}$. Note that the maps $r_{n}$
have the following uniform bound in $\mathbb{C}^{+}$: 
\begin{equation}
\left|r_{n}(z)\right|\leq\frac{2c}{n\sqrt{n}\Im z}+\frac{1}{n^{2}(\Im z)^{2}},\quad z\in\mathbb{C}^{+},\quad n\geq1,\label{cbound}
\end{equation}
if we assume 
\[
c=\int_{\mathbb{R}}\left|t\right|\,d\nu(t)\in[0,+\infty).
\]

Below is the main result of this section, and Theorem 1.1 is the part
(1) in it.
\begin{thm}
Let $\mu$ be a probability measure on $\mathbb{R}$ with $m(\mu)=0$
and $var(\mu)=1$, and let $\mu_{n}$ be the normalization of $\mu^{\triangleright n}$
by the factor of $n^{-1/2}$. Assume the first absolute moment $c$
of the measure $\nu$ is finite. 
\begin{enumerate}
\item If $c>0$ then
\[
\left\Vert C_{\mu_{n}}-C_{\gamma}\right\Vert _{\infty}\leq\frac{71\sqrt[4]{c}}{n^{1/8}},\quad n>\max\left\{ (8\sqrt{3c})^{4},(8c\sqrt{c})^{-4}\right\} .
\]
In particular, this convergence rate holds when the fourth moment
$m_{4}(\mu)<+\infty$.
\item If $c>0$ and if the $F$-transform of $\nu$ has the Nevanlinna form
\[
F_{\nu}(z)=z-m(\nu)+\int_{\mathbb{R}}\frac{1}{t-z}\,d\rho(t)
\]
where  
\[
d=\int_{\mathbb{R}}\left|t\right|\,d\rho(t)\in[0,+\infty),
\]
then the estimate 
\[
\left\Vert C_{\mu_{n}}-C_{\gamma}\right\Vert _{\infty}\leq\frac{200\sqrt{d+3(1+m_{2}(\nu))^{2}}}{n^{1/4}}
\]
holds for $n>\max\left\{ 4m(\nu)^{2},4m_{2}(\nu)^{2},12288\left[d+3(1+m_{2}(\nu))^{2}\right]^{2}\right\} $.
In particular, this convergence rate holds if $m_{6}(\mu)<+\infty$.
\item The case $c=0$ corresponds to $\mu=(\delta_{-1}+\delta_{1})/2$,
and we have 
\[
\left\Vert C_{\mu_{n}}-C_{\gamma}\right\Vert _{\infty}\leq\frac{200}{n^{1/4}},\quad n>12288.
\]
\end{enumerate}
\end{thm}

\begin{proof}
We first prove (1). Define the error function
\[
\varepsilon_{n}(z)=\sum_{j=0}^{n-1}r_{n}\left(F_{n}^{\circ j}(z)\right),\quad z\in\mathbb{C}^{+},
\]
and assume $n>\max\left\{ (8\sqrt{3c})^{4},(8c\sqrt{c})^{-4}\right\} $.
From the uniform bound (\ref{cbound}), we know that $\left|\varepsilon_{n}(z)\right|\leq2cn^{-1/2}y^{-1}+n^{-1}y^{-2}<3y/2$
for all $z$ on the line
\[
\left\{ x+iy:x\in\mathbb{R},\:y=2\sqrt{c}n^{-1/4}\right\} ,
\]
and hence the square root $\sqrt{z^{2}-2+\varepsilon_{n}(z)}$ is
well-defined for such $n$ and $z$. Furthermore, the computation
(\ref{iteration}) yields the formula 
\[
F_{\mu_{n}}(z)=\sqrt{z^{2}-2+\varepsilon_{n}(z)}.
\]
Therefore, Lemmas 2.1 and 2.2 allow us to conclude that 
\[
\left\Vert C_{\mu_{n}}-C_{\gamma}\right\Vert _{\infty}\leq50\sqrt{y}\leq71\sqrt[4]{c}\,n^{-1/8}.
\]

Next, we prove (2) by improving the bound of $\varepsilon_{n}$ on
the horizontal line
\[
L_{n}=\left\{ x+iy:x\in\mathbb{R},\;y=16\left[d+3(1+m_{2}(\nu))^{2}\right]n^{-1/2}\right\} .
\]

For computational convenience, we introduce the notation 
\[
G_{\tau}(z)=\int_{\mathbb{R}}\frac{1}{z-t}\,d\tau(t)
\]
for any finite Borel measure $\tau$ on $\mathbb{R}$, and recall
the two Nevanlinna representaions
\[
F_{\mu}(z)=z-G_{\nu}(z)\quad\text{and}\quad F_{\nu}(z)=z-m(\nu)-G_{\rho}(z),
\] where $\rho(\mathbb{R})=var(\nu)$.
Furthermore, we set $\widehat{G_{\rho}}(z)=m(\nu)+G_{\rho}(z)$. Note that although the map $G_{\rho}$ may be constantly zero on $\mathbb{C}^{+}$ when $var(\nu)=0$, the map $\widehat{G_{\rho}}$ never vanishes in $\mathbb{C}^{+}$ because $c>0$. Accordingly,
the error $\varepsilon_{n}(z)$ now becomes 
\[
\varepsilon_{n}(z)=\underbrace{-\frac{2}{n}\sum_{j=0}^{n-1}\frac{\widehat{G_{\rho}}\left(\sqrt{n}F_{n}^{\circ j}(z)\right)}{\sqrt{n}F_{n}^{\circ j}(z)-\widehat{G_{\rho}}\left(\sqrt{n}F_{n}^{\circ j}(z)\right)}}_{=A_{n}(z)}+\underbrace{\frac{1}{n}\sum_{j=0}^{n-1}\frac{1}{\left[\sqrt{n}F_{n}^{\circ j}(z)-\widehat{G_{\rho}}\left(\sqrt{n}F_{n}^{\circ j}(z)\right)\right]^{2}}}_{=B_{n}(z)}.
\]

We first estimate $B_{n}(z)$. Note that 
\begin{equation}
\left|F_{n}(z)-z\right|=\left|\frac{1}{\sqrt{n}}\int_{\mathbb{R}}\frac{1}{t-\sqrt{n}z}\,d\nu(t)\right|\leq1,\quad\Im z\geq n^{-1},\quad n\geq1.\label{Bound1}
\end{equation}
Meanwhile, recall the dilation equation 
\[
F_{n}(z)=z-\frac{1}{\sqrt{n}}G_{\nu}(\sqrt{n}z)=z-\frac{1}{n}\frac{1}{z-m(\nu)/\sqrt{n}-G_{\rho}\left(\sqrt{n}z\right)/\sqrt{n}},
\]
and observe from which that 
\begin{equation}
\left|F_{n}(z)-z\right|\leq n^{-1},\quad\Im z\geq n^{-1/2},\quad\left|z\right|\geq2,\quad n>\max\left\{ 4m(\nu)^{2},4m_{2}(\nu)^{2}\right\} .\label{Bound2}
\end{equation}

Now for any $z\in L_{n}$ with $\left|z\right|\geq3$, the inequality
(\ref{Bound2}) yields $\left|F_{n}(z)-z\right|\leq n^{-1}$ for sufficiently
large $n$. It follows that $\left|F_{n}(z)\right|\geq\left|z\right|-n^{-1}>2$
and $\Im F_{n}(z)>\Im z\geq n^{-1/2}$. Therefore one can apply (\ref{Bound2})
to $F_{n}(z)$ to get \[\left|F_{n}^{\circ2}(z)-z\right|\leq\left|F_{n}\left(F_{n}(z)\right)-F_{n}(z)\right|+\left|F_{n}(z)-z\right|\leq2n^{-1}.\] Proceeding inductively, we obtain 
\[
\left|F_{n}^{\circ n}(z)-z\right|\leq1,\quad z\in L_{n},\quad\left|z\right|\geq3,\quad n>\max\left\{ 4m(\nu)^{2},4m_{2}(\nu)^{2}\right\} .
\]
As for $z\in L_{n}$, $\left|z\right|<3$, we introduce the finite
set $J=\left\{ j\in\mathbb{N}\cap[1,n]:\left|F_{n}^{\circ j}(z)\right|\geq3\right\} $
and observe that if $J=\phi$, then the orbit $\{F_{n}^{\circ j}(z)\}_{j=0}^{n}$
stays within the ball $\left\{ |z|<3\right\} $ all the time, whence
\[
\left|F_{n}^{\circ n}(z)-z\right|\leq\left|F_{n}^{\circ n}(z)\right|+\left|z\right|<6.
\]
If $J\neq\phi$, we take $k=\min J$, the first exit time of the orbit
$\{F_{n}^{\circ j}(z)\}_{j=0}^{n}$ out of the ball, then (\ref{Bound1})
and (\ref{Bound2}) together imply that 
\begin{multline*}
\left|F_{n}^{\circ n}(z)-z\right|\leq\left|F_{n}^{\circ(n-k)}\left(F_{n}^{\circ k}(z)\right)-F_{n}^{\circ k}(z)\right|\\
+\left|F_{n}\left(F_{n}^{\circ(k-1)}(z)\right)-F_{n}^{\circ(k-1)}(z)\right|+\left|F_{n}^{\circ(k-1)}(z)-z\right|\leq\frac{n-k}{n}+1+6<8
\end{multline*}
Thus, in all cases, we always have 
\[
\left|F_{n}^{\circ n}(z)-z\right|<8,\quad z\in L_{n},\quad n>\max\left\{ 4m(\nu)^{2},4m_{2}(\nu)^{2}\right\} .
\]

By running an induction argument on $F_{n}(z)=z-n^{-1/2}G_{\nu}(\sqrt{n}z)$,
we get
\[
F_{n}^{\circ n}(z)=z-n^{-1/2}\sum_{j=0}^{n-1}G_{\nu}\left(\sqrt{n}F_{n}^{\circ j}(z)\right).
\]
Hence, the preceding estimate on $F_{n}^{\circ n}(z)$ leads to 
\[
\left|\sum_{j=0}^{n-1}G_{\nu}\left(\sqrt{n}F_{n}^{\circ j}(z)\right)\right|<8\sqrt{n}.
\]
Since $G_{\nu}(z)=1/[z-\widehat{G_{\rho}}(z)]$, $\Im\widehat{G_{\rho}}(z)=\Im G_{\rho}(z)\leq0$,
and $\Im F_{n}(z)>\Im z$, we conclude further that 
\begin{eqnarray}
8\sqrt{n}>\left|\sum_{j=0}^{n-1}G_{\nu}\left(\sqrt{n}F_{n}^{\circ j}(z)\right)\right| & \geq & \Im\sum_{j=0}^{n-1}G_{\nu}\left(\sqrt{n}F_{n}^{\circ j}(z)\right)\label{Gnu}\\
 & = & \sum_{j=0}^{n-1}\frac{\sqrt{n}\Im F_{n}^{\circ j}(z)-\Im G_{\rho}\left(\sqrt{n}F_{n}^{\circ j}(z)\right)}{\left|\sqrt{n}F_{n}^{\circ j}(z)-\widehat{G_{\rho}}\left(\sqrt{n}F_{n}^{\circ j}(z)\right)\right|^{2}}\nonumber \\
 & \geq & \sum_{j=0}^{n-1}\frac{1}{\left|\sqrt{n}F_{n}^{\circ j}(z)-\widehat{G_{\rho}}\left(\sqrt{n}F_{n}^{\circ j}(z)\right)\right|^{2}}\nonumber \\
 & \geq & n\left|B_{n}(z)\right|.\nonumber 
\end{eqnarray}
In summary, we have shown that 
\[
\left|B_{n}(z)\right|<8n^{-1/2},\quad z\in L_{n},\quad n>\max\left\{ 4m(\nu)^{2},4m_{2}(\nu)^{2}\right\} .
\]

Next, we turn to $A_{n}(z)$. For $z\in L_{n}$ and $n>\max\left\{ 4m(\nu)^{2},4m_{2}(\nu)^{2}\right\} $,
(\ref{Gnu}) and the following two inequalities 
\begin{eqnarray*}
\left|G_{\rho}\left(\sqrt{n}F_{n}^{\circ j}(z)\right)\right|\left|\widehat{G_{\rho}}\left(\sqrt{n}F_{n}^{\circ j}(z)\right)\right| & \leq & \frac{\rho(\mathbb{R})}{\sqrt{n}\Im F_{n}^{\circ j}(z)}\left[\left|m(\nu)\right|+\frac{\rho(\mathbb{R})}{\sqrt{n}\Im F_{n}^{\circ j}(z)}\right]\\
 & \leq & m_{2}(\nu)\left|m(\nu)\right|+m_{2}(\nu)^{2},
\end{eqnarray*}
\[
\left|\sqrt{n}F_{n}^{\circ j}(z)\right|\left|G_{\rho}\left(\sqrt{n}F_{n}^{\circ j}(z)\right)\right|=\left|\int_{\mathbb{R}}\frac{t}{\sqrt{n}F_{n}^{\circ j}(z)-t}\,d\rho(t)+\rho(\mathbb{R})\right|\leq d+m_{2}(\nu)
\]
imply that 
\begin{eqnarray*}
\left|\frac{n}{2}A_{n}(z)\right| & \leq & \left|\sum_{j=0}^{n-1}\frac{m(\nu)}{\sqrt{n}F_{n}^{\circ j}(z)-\widehat{G_{\rho}}\left(\sqrt{n}F_{n}^{\circ j}(z)\right)}\right|\\
 &  & +\left|\sum_{j=0}^{n-1}\frac{G_{\rho}\left(\sqrt{n}F_{n}^{\circ j}(z)\right)}{\sqrt{n}F_{n}^{\circ j}(z)-\widehat{G_{\rho}}\left(\sqrt{n}F_{n}^{\circ j}(z)\right)}\right|\\
 & = & \left|m(\nu)\right|\left|\sum_{j=0}^{n-1}G_{\nu}\left(\sqrt{n}F_{n}^{\circ j}(z)\right)\right|\\
 &  & +\left|\sum_{j=0}^{n-1}\frac{G_{\rho}\left(\sqrt{n}F_{n}^{\circ j}(z)\right)\overline{\left[\sqrt{n}F_{n}^{\circ j}(z)-\widehat{G_{\rho}}\left(\sqrt{n}F_{n}^{\circ j}(z)\right)\right]}}{\left|\sqrt{n}F_{n}^{\circ j}(z)-\widehat{G_{\rho}}\left(\sqrt{n}F_{n}^{\circ j}(z)\right)\right|^{2}}\right|\\
 & \leq & 8\left|m(\nu)\right|\sqrt{n}+8\left[d+m_{2}(\nu)+m_{2}(\nu)\left|m(\nu)\right|+m_{2}(\nu)^{2}\right]\sqrt{n}\\
 & \leq & 8\left[d+3(1+m_{2}(\nu))^{2}\right]\sqrt{n}.
\end{eqnarray*}
So we have 
\[
\left|A_{n}(z)\right|\leq16\left[d+3(1+m_{2}(\nu))^{2}\right]n^{-1/2},\quad z\in L_{n},\quad n>\max\left\{ 4m(\nu)^{2},4m_{2}(\nu)^{2}\right\} .
\]

Finally, the hypotheses 
\[
\left|\varepsilon_{n}(z)\right|<24\left[d+3(1+m_{2}(\nu))^{2}\right]n^{-1/2}=3y/2
\]
and $y<(4\sqrt{3})^{-1}$ are satisfied if $z\in L_{n}$ and 
\[
n>\max\left\{ 4m(\nu)^{2},4m_{2}(\nu)^{2},12288\left[d+3(1+m_{2}(\nu))^{2}\right]^{2}\right\}.
\]
Therefore, the convergence rate in (2) follows from Lemmas 2.1 and
2.2 again.

Finally, (3) can be proved by examining the values of $\left|\varepsilon_{n}(z)\right|$
on the line 
\[
L_{n}=\left\{ x+iy:x\in\mathbb{R},\:y=16n^{-1/2}\right\} .
\]
Indeed, in this case we have $\nu=\delta_0$ and
\[
\varepsilon_{n}(z)=\frac{1}{n^{2}}\sum_{j=0}^{n-1}\frac{1}{F_{n}^{\circ j}(z)^{2}},
\]
so that a similar argument as in the proof of (2) shows that $\left|\varepsilon_{n}\right|<8n^{-1/2}$
on $L_{n}$ for all $n\geq1$. We omit the details to avoid repetition. 
\end{proof}
\begin{rem}
We associate each parameter $r>0$ with a probability measure $\nu_{r}$
whose $F$-transform is given by 
\[
F_{\nu_{r}}(z)=r+\sqrt{(z-r)^{2}-2}.
\]
Since $-\Im G_{\nu_{r}}(x+iy)/\pi$ is the Poisson integral of $\nu_{r}$
in $\mathbb{C}^{+}$, it is easy to see from the boundary values $\lim_{y\rightarrow0^{+}}\Im G_{\nu_{r}}(x+iy)$
that $\nu_{r}$ is supported on the disjoint union
\[
\left\{ -\sqrt{2+r^{2}}+r\right\} \cup[-\sqrt{2}+r,\sqrt{2}+r],
\]
and the mass of $\nu_{r}$ at $-\sqrt{2+r^{2}}+r$ is $|r|/(2+r^{2})$.
We now consider the measure $\mu=\nu_{1}$ and note, by induction,
that 
\[
F_{\mu_{n}}(z)=n^{-1/2}+\sqrt{(z-n^{-1/2})^{2}-2}=F_{\nu_{1/\sqrt{n}}}(z).
\]
Since $-\sqrt{2}+(4n)^{-1/2}<-\sqrt{2+n^{-1}}+n^{-1/2}$, it follows
that the interval $(-\infty,-\sqrt{2}+(4n)^{-1/2}]$ is disjoint from
the support of $\mu_{n}$ and hence 
\begin{eqnarray*}
\left\Vert C_{\mu_{n}}-C_{\gamma}\right\Vert _{\infty} & \geq & \left|\mu_{n}\left((-\infty,-\sqrt{2}+(4n)^{-1/2}])\right)-\gamma\left((-\infty,-\sqrt{2}+(4n)^{-1/2}])\right)\right|\\
 & = & \left|\gamma\left([-\sqrt{2},-\sqrt{2}+(4n)^{-1/2}])\right)\right|\\
 & = & \int_{-\sqrt{2}}^{-\sqrt{2}+(4n)^{-1/2}}\frac{dt}{\pi\sqrt{2-t^{2}}}\geq\frac{1}{2\pi}\int_{-\sqrt{2}}^{-\sqrt{2}+(4n)^{-1/2}}\frac{dt}{\sqrt{\sqrt{2}+t}}=\frac{1}{\sqrt{2}\pi n^{1/4}}.
\end{eqnarray*}
This shows that the convergence rate in Theorem 3.1 (2) cannot be
improved without further assumptions on the measure $\mu$. 
\end{rem}

\section{Return Sequence }

In addition to having zero mean and unit variance, we now assume that
$\mu$ is also singular to the Lebesgue measure $\lambda$, so that the
boundary restriction $T_{\mu}x=\lim_{y\rightarrow0^{+}}F_{\mu}(x+iy)\in\mathbb{R}$
for $\lambda$-a.e. $x\in\mathbb{R}$. In \cite{AaronsonPaper}, Aaronson
proved the following formula
\[
a_{n}(T_{\mu})\sim\frac{1}{\pi}\sum_{j=1}^{n}\Im\frac{-1}{F_{\mu}^{\circ j}(z)},\quad z\in\mathbb{C}^{+}.
\]
By taking an appropriate $z$ on the $y$-axis, we will show that the series on the right side is asymptotically equivalent to $\sqrt{2n}/\pi$ as $n\rightarrow \infty$.

Toward this end, we first prove a result that is of some interest in
complex dynamics. It is rather obvious that the sequence $\{F_{\mu}^{\circ j}\}_{j=1}^{\infty}$ tends to its Denjoy-Wolff point $\infty$ pointwisely on $\mathbb{C}^{+}$. Under our assumptions on mean
and variance of $\mu$, this convergence is in fact non-tangential
for points that are sufficiently far from the real line. 
\begin{prop}
Recall the Nevanlinna representation $F_{\mu}(z)=z-G_{\nu}(z)$. Let
$k$ be a positive integer such that $\nu\left([-k,k]\right)\geq0.9$
and define the truncated cone
\[
\Gamma=\left\{ x+iy:\left|x\right|\leq y,\;y\geq2k+2\right\}.
\]
Then $\Gamma$ is an invariant set under the map $F_{\mu}$, that
is, $F_{\mu}\left(\Gamma\right)\subset\Gamma$. Consequently, the
parabolic iterations $F_{\mu}^{\circ j}(z)$ tend to $\infty$ non-tangentially
for every $z\in\Gamma$.
\end{prop}

\begin{proof}
We decompose the cone $\Gamma$ into a disjoint union $\Gamma=\Gamma_{-}\cup\Gamma_{0}\cup\Gamma_{+}$,
where 
\[
\Gamma_{-}=\left\{ x+iy\in\Gamma:-y\leq x\leq-y+1\right\} 
\]
and 
\[
\Gamma_{+}=\left\{ x+iy\in\Gamma:y-1\leq x\leq y\right\} .
\]

It follows that $\Re G_{\nu}(z)>0$ for $z\in\Gamma_{+}$ and that
$\Re G_{\nu}(z)<0$ for $z\in\Gamma_{-}$. Indeed, since 
\[
\frac{1}{2}\leq\theta(t)=\frac{x-t}{y}\leq\frac{3}{2}
\]
holds for any $z=x+iy\in\Gamma_{+}$ and $t\in[-k,k]$, we have 
\begin{eqnarray*}
\Re G_{\nu}(z) & = & \frac{1}{y}\int_{\left|t\right|\leq k}\frac{\theta(t)}{\theta(t)^{2}+1}\,d\nu(t)+\frac{1}{y}\int_{\left|t\right|>k}\frac{\theta(t)}{\theta(t)^{2}+1}\,d\nu(t)\\
 & \geq & \frac{1}{y}\int_{\left|t\right|\leq k}\frac{2}{5}\,d\nu(t)-\frac{1}{y}\int_{\left|t\right|>k}\frac{1}{2}\,d\nu(t)\\
 & = & \frac{1}{y}\left[0.9\nu\left([-k,k]\right)-0.5\right]>0.
\end{eqnarray*}
The proof of $\Re G_{\nu}(z)<0$ for $z\in\Gamma_{-}$ is similar.

We now consider any $z=x+iy\in\Gamma$. If $z\in\Gamma_{0}$ then
the distance from $z$ to the exterior $\left\{w\in\mathbb{C}^{+}:\Im w\geq 2k+2,\;w\notin\Gamma\right\}$
is at least $1/\sqrt{2}$. Since $\left|F_{\mu}(z)-z\right|=\left|G_{\nu}(z)\right|\leq y^{-1}<2^{-1}$,
we have $F_{\mu}(z)\in\Gamma$. Next, if $z\in\Gamma_{+}$, we observe
that 
\[
\Re F_{\mu}(z)=x-\Re G_{\nu}(z)<x\leq y\leq\Im F_{\mu}(z)
\]
and 
\[
\Re F_{\mu}(z)=x-\Re G_{\nu}(z)\geq y-1-\frac{1}{y}>\frac{1}{2}>-\Im F_{\mu}(z).
\]
So we also have $F_{\mu}(z)\in\Gamma$ in this case. Finally, the
case of $z\in\Gamma_{-}$ follows similarly from the fact that $\Re G_{\nu}<0$
on $\Gamma_{-}$. 
\end{proof}
We proceed to the main result, from which Theorem 1.2 follows.
\begin{thm}
For $z=(2k+2)i$, we have 
\[
\sum_{j=1}^{n}\Im\frac{-1}{F_{\mu}^{\circ j}(z)}\sim\sqrt{2n}\quad(n\rightarrow\infty),
\]
whence $a_{n}(T_{\mu})\sim\sqrt{2n}/\pi.$
\end{thm}

\begin{proof}
Denoting $z_{j}=F_{\mu}^{\circ j}(z)$ for $j\geq1$, the Nevanlinna
representation of $F_{\mu}$ implies 
\[
z_{j+1}-z_{j}=\int_{\mathbb{R}}\frac{1}{t-z_{j}}\,d\nu(t),
\]
so that 
\[
\left|\frac{z_{j+1}}{z_{j}}-1\right|\leq\frac{1}{\left|z_{j}\right|}\int_{\mathbb{R}}\frac{1}{|t-z_{j}|}\,d\nu(t)\leq\frac{1}{|z_{j}|(2k+2)}\rightarrow0\quad(j\rightarrow\infty).
\]

Proposition 4.1 yields
\[
\left|\frac{z_{j}}{t-z_{j}}\right|\leq\sqrt{1+\left(\frac{\Re z_{j}}{\Im z_{j}}\right)^{2}}\leq\sqrt{2},\quad t\in\mathbb{R},
\]
and hence the dominated convergence theorem shows further that 
\[
\lim_{j\rightarrow\infty}\int_{\mathbb{R}}\frac{z_{j}}{t-z_{j}}\,d\nu(t)=-1
\]

Now we have 
\[
z_{j+1}^{2}-z_{j}^{2}=\left[\frac{z_{j+1}}{z_{j}}+1\right]\int_{\mathbb{R}}\frac{z_{j}}{t-z_{j}}\,d\nu(t)\rightarrow-2
\]
as $j\rightarrow\infty$. Consequently, we obtain the convergence
of the averages 
\[
\frac{1}{j-1}\sum_{k=1}^{j-1}\left(z_{k+1}^{2}-z_{k}^{2}\right)\rightarrow-2\qquad(j\rightarrow\infty).
\]
Observe next that 
\[
\frac{z_{j}^{2}}{j}=\frac{z_{1}^{2}}{j}+\frac{j-1}{j}\left[\frac{1}{j-1}\sum_{k=1}^{j-1}\left(z_{k+1}^{2}-z_{k}^{2}\right)\right],\quad j\geq2,
\]
and so we finally get 
\[
\lim_{j\rightarrow\infty}\frac{z_{j}^{2}}{j}=-2.
\]

Therefore, for sufficiently large $j$, we can apply the analytic
square root $\sqrt{\cdot}$ to $z_{j}^{2}/j$, and the continuity
of $\sqrt{\cdot}$ on $(-\infty,0)$ says that 
\[
\lim_{j\rightarrow\infty}\frac{z_{j}}{\sqrt{j}}=i\sqrt{2}.
\]
This shows that 
\[
\Im\frac{-1}{z_{j}}\sim\frac{1}{\sqrt{2j}}\qquad(j\rightarrow\infty),
\]
and naturally, 
\[
\sum_{j=1}^{n}\Im\frac{-1}{z_{j}}\sim\sqrt{2n}\qquad(n\rightarrow\infty).
\]
\end{proof}
We conclude this paper with the following remarks. First, a simple normalization
argument shows that
\[
\sum_{j=1}^{n}\Im\frac{-1}{F_{\mu}^{\circ j}((2k+2)i)}\sim\sqrt{\frac{2n}{var(\mu)}}\quad(n\rightarrow\infty)
\]
for sufficiently large $k\in\mathbb{N}$, as long as $m(\mu)=0$ and
$var(\mu)\in(0,+\infty)$. Also, the singularity of $\mu$ to the
Lebesgue measure plays no role in our proofs of Proposition 4.1 and
Theorem 4.2. Finally, we mention that Theorem 4.2 also implies other
ergodic properties of $T_{\mu}$ such as log-lower boundedness and
quasi-finiteness (cf. \cite[Proposition 3.1]{AaronsonPark}). 

\section*{acknowledgement}

The third author was supported by the NSERC Canada Discovery Grant
RGPIN-2016-03796.

\end{document}